\documentclass[reqno]{amsart}
\usepackage{hyperref}

\begin{document}
\title[\hfilneg \hfil Three solutions for a fractional Laplacian]
{Three solutions for a fractional elliptic problems with critical and supercritical growth }

\author[J. Zhang \hfil \hfilneg]
{Jinguo Zhang}  

\address{Jinguo Zhang \newline
School of Mathematics, Jiangxi Normal University,
Nanchang 330022, China}
\email{jgzhang@jxnu.edu.cn}

\thanks{}

\begin{abstract}
In this paper, we deal with the existence and multiplicity of solutions for the
fractional elliptic problems involving critical and supercritical Sobolev
exponent via variational arguments. By means of the truncation combining with
the Moser iteration, we prove that the problems has at least three solutions.

\end{abstract}

\maketitle
\numberwithin{equation}{section}
\newtheorem{theorem}{Theorem}[section]
\newtheorem{lemma}[theorem]{Lemma}
\newtheorem{definition}[theorem]{Definition}
\newtheorem{remark}[theorem]{Remark}
\allowdisplaybreaks

\section{Introduction and main result}
\section{Introduction}
In  this paper, we consider the existence and multiplicity of solutions for the fractional elliptic problem
\begin{equation}\label{eq01}
\left\{
\aligned
&(-\Delta)^{s}u=\lambda f(x,u)+\mu |u|^{p-2}u,\quad &\text{in}\,\, \Omega,\\
&u=0,&\text{on}\,\,\partial\Omega,
\endaligned\right.
\end{equation}
where $\Omega\subset \mathbb{R}^{N}$, $N\geq 2$, is a smooth bounded domain, $(-\Delta)^{s}$ stands for the fractional Laplacian,
$p\geq 2^{*}_{s}=\frac{2N}{N-2s}$, $\mu$ and  $\lambda $
are nonnegative constants and
 $f:\Omega\times\mathbb{R}\to \mathbb{R}$ is a Car\'{a}theodory function.

 The fractional Laplacian appears in diverse areas including physics,
 biological modeling, mathematical finances and especially partial
 differential equations involving the fractional Laplacian have been attracted by
 researchers. An important feature of the fractional Laplacian is its nonlocal property,
 which makes it difficult to handle.  Recently, Caffarelli and Silvestre \cite{ll2007}
 developed a local interpretation of the fractional Laplacian through Dirichlet-Neumann maps.
 This is commonly used in the recent literature since it allows to write nonlocal
 problems in a local way and this permits to use the variational methods for these
kinds of problems.

Based on these extensions, many authors studied nonlinear problem of the form
$(-\Delta)^{s}=f(x,u)$ for a certain function $f:\mathbb{R}^{N}\to \mathbb{R}$.
 Among others, it is worthwhile to mention the work by
Cabr\'{e}-Tan \cite{xt2010} and Tan \cite{t2011} when $s=\frac{1}{2}$. They established the
existence of positive solutions for equations having the subcritical growth, their regularity
and symmetry properties. Recently, and for the subcritical case, Choi, Kim and Lee \cite{ckl2013}
developed a nonlocal analog of the results by Han \cite{han1991} and Rey \cite{rey1990}.

In this paper, we study the existence and multiplicity of solutions for the problem with critical
and supercritical growth. For our problem, the first difficulty lies in
that the fractional Laplacian operator $(-\Delta)^{s}$ is nonlocal,
the nonlocal property of $(-\Delta)^{s}$ makes some calculations  difficult. To overcome this difficult,
we do not work on the space $H^{s}_{0}(\Omega)$ directly, we transform the nonlocal problem into a local problem by the extension introduced by Caffarelli and Silvestre in \cite{ll2007}.
After this extension, the problem \eqref{eq01} can be reduced to the problem
\begin{equation}\label{eq02}
\left\{
\aligned
&div(y^{1-2s}\nabla w)=0,\quad &\text{in}\,\,\mathcal{C},\\
&w=0,&\text{on}\,\,\partial_{L}\mathcal{C},\\
&\partial^{s}_{\nu}w=\lambda f(x,w)+\mu |w|^{p-2}w,\quad &\text{in}\,\, \Omega\times\{0\},\\
\endaligned\right.
\end{equation}
where $\nu$ is the outward unit normal vector to $\mathcal{C}$ on $\Omega\times\{0\}$ and
$$\partial^{s}_{\nu}w(x,0):=-\lim\limits_{y\to 0^+}y^{1-2s}\frac{\partial w}{\partial y}(x,y),\quad \forall x\in \Omega.$$
Obviously, the equation \eqref{eq02} is a local problem.

The second difficult lies in that problem \eqref{eq01} is a supercritical problem.
Hence, we can not use directly the variational techniques because the
corresponding energy functional is not well-defined on the Sobolev space $H^{s}_{0}(\Omega)$.
To overcome this difficult, one usually uses the truncation and the Moser iteration. This spirt
has been widely used in the supercritical Laplacian equation in the past few decades,
see \cite{jj1997,abc1994,j1960,jsg2006,ff2007,zz2014} and references therein.

The aim of this paper is to study the problem \eqref{eq01} when $p\geq 2^{*}_{s}$. During this
study we develop some nonlocal techniques which also have their own interests.
In order to state our main results, we formulate the following assumptions:
\begin{itemize}
\item[$(f_1)$] $\lim\limits_{|t|\to +\infty}\frac{f(x,t)}{|t|}=0$ uniformly in $x\in \Omega$;
\item[$(f_2)$]$\lim\limits_{|t|\to 0}\frac{f(x,t)}{|t|}=0$ uniformly in $x\in \Omega$;
\item[$(f_3)$] $\sup\limits_{u\in H^{s}_{0}(\Omega)}\int_{\Omega} F(x,u)dx>0$, and for every $M>0$,
$f(x,u)\in L^{\infty}(\Omega)$ for each $|u|\leq M$,
 where $F(x,u)=\int_{0}^{u}f(x,t)dt$.
\end{itemize}

Set
$$\theta:=\frac{1}{2}\inf\Big{\{}\frac{\int_{\Omega}|(-\Delta)^{\frac{s}{2}}u|^{2}dx}{\int_{\Omega}F(x,u)dx}:\,\,u\in H^{s}_{0}(\Omega),\int_{\Omega}F(x,u)dx>0\Big{\}}.$$
The main results are as follows.
\begin{theorem}\label{th1}
Assume that $(f_1)-(f_3)$ hold. Then there exists  $\delta>0$
such that for any $\mu\in [0\,,\,\delta]$, there exist an
compact interval $[a\,,\,b]\subset(\frac{1}{\theta}\,,\,+\infty)$ and a
constant $\gamma>0$ such that for each $\lambda\in[a\,,\,b]$, the
problem \eqref{eq01} has at least three solutions in $H^{s}_{0}(\Omega)$, whose norms are less than $\gamma$.
\end{theorem}

For the general problem
\begin{equation}\label{eq04}
\left\{
\aligned
&(-\Delta)^{s}u=\lambda f(x,u)+\mu g(x,u),\quad &\text{in}\,\,\Omega,\\
&u=0,&\text{on}\,\, \partial\Omega,
\endaligned\right.
\end{equation}
where $\Omega\subset \mathbb{R}^{N}$ is a bounded smooth domain, and
\begin{itemize}
\item[(g)]$|g(x,u)|\leq C(1+|u|^{p-1})$, where $ p\geq 2^*_{s}=\frac{2N}{N-2s}$, $ C>0$.
\end{itemize}
If $f$ satisfies the conditions $(f_1)-(f_3)$, we also have the following result similar to Theorem \ref{th1}.

\begin{theorem}\label{th2}
Let $f$ satisfy $(f_1)-(f_3)$ and $g$ satisfy $(g)$. Then  there exists  $\delta>0$
such that for any $\mu\in [0\,,\,\delta]$, there exist an
compact interval $[a\,,\,b]\subset(\frac{1}{\theta}\,,\,+\infty)$ and a
constant $\gamma>0$ such that for each $\lambda\in[a\,,\,b]$, the
problem \eqref{eq04} has at least three solutions in $H^{s}_{0}(\Omega)$,
whose $H^{s}_{0}(\Omega)$-norms are less than $\gamma$.
\end{theorem}

The paper is organized as follows. In Section 2, we introduce a variational
setting of the problem and present some preliminary results.
In Section 3, some properties of the fractional operator are discussed, and
apply the truncation and the Moser iteration to obtain
  the proof of Theorem \ref{th1} and Theorem \ref{th2}.

For convenience we fix some notations. $L^{p}(\Omega)$ $(1<p\leq \infty)$
denotes the usual Sobolev space with norm $\|\cdot\|_{L^p}$;
$C_{0}(\bar{\Omega})$ denotes the space of continuous real functions in
 $\bar{\Omega}$ vanishing on the boundary $\partial\Omega$;
 $C$ or $C_{i}(i=1,2,\cdot\cdot\cdot,)$ denote any positive constant.

\section{Preliminaries and functional setting}
In this section we recall some basic properties of the fractional Laplacian.
In the entire space, the operator $(-\Delta)^{s}$ in $\mathbb{R}^{N}$, $0<s<1$,
is defined through Fourier transform $\mathcal{F}$, by
$$\mathcal{F}[(-\Delta)^{s}u](\xi)=|\xi|^{2s}\mathcal{F}[u](\xi).$$
On a bounded domain $\Omega$, we define $(-\Delta)^{s}$ through the spectral decomposition
of $-\Delta$ in $H^{1}_{0}(\Omega)$:
$$(-\Delta)^{s}u=\sum\limits_{i=1}^{\infty}\mu^{s}_{i}u_{i}\varphi_{i},$$
where $u=\sum\limits_{i=1}^{\infty}u_{i}\varphi_{i}$, $u_{i}=\int_{\Omega}u\varphi_{i}dx$
and
$(\mu_{i}\,,\,\varphi_{i})$ are the eigenvalues and corresponding eigenfunctions
of $-\Delta$ on $H^{1}_{0}(\Omega)$. The fractional Laplacian is well defined in
the fractional Sobolev space $H^{s}_{0}(\Omega)$,
$$H^{s}_{0}(\Omega)=\{u=\sum a_{j}\varphi_{j}\in L^{2}(\Omega):\|u\|_{H^{s}_{0}}=(\sum a_{j}^{2}\lambda^{s})^{\frac{1}{2}}<\infty\},$$
which is a Hilbert space endowed with the following inner product
$$\langle\sum\limits_{i=1}^{\infty}a_{i}\phi_{i}\,,\,\sum\limits_{i=1}^{\infty}b_{i}\phi_{i}\rangle=
\sum\limits_{i=1}^{\infty}a_{i}b_{i}\mu_{i}^{s},$$
and we have the following expression for this inner product
$$\langle u\,,\,v\rangle=\int\limits_{\Omega}(-\Delta)^{\frac{s}{2}}u\cdot(-\Delta)^{\frac{s}{2}}vdx
=\int\limits_{\Omega}(-\Delta)^{s}u\,v\,dx,\quad \forall \,u,\,v\in H^{s}_{0}(\Omega).$$

We will often work with an equivalent definition based on an appropriate extension problem
introduced by Caffaarelli and Silvestre \cite{ll2007}. Let $\Omega$ be a bounded smooth domain in $\mathbb{R}^{N}$, the half cylinder with base $\Omega$ denote by
$\mathcal{C}=\Omega\times(0\,,\,\infty)\subset \mathbb{R}^{N+1}_{+}$ and
its lateral boundary given that $\partial_{L}\mathcal{C}=\partial\Omega\times[0\,,\,\infty)$, where
$$\mathbb{R}^{N+1}_{+}=\{(x,y)=(x_{1},x_{2},\cdot\cdot\cdot,x_{n},y)\in \mathbb{R}^{N+1}:\,y>0\}.$$

The space $H^{s}_{0,L}(\mathcal{C})$ is defined as the completion of
$$C^{s}_{0,L}(\mathcal{C})=\{w\in C^{\infty}(\bar{\mathcal{C}}):
\,w=0\,\,\text{on}\,\,\partial_{L}\mathcal{C}\}$$
with respect to the norm
$$\|w\|_{H^{s}_{0,L}(\mathcal{C})}=\Big(\int\limits_{\mathcal{C}}y^{1-2s}|\nabla w|^{2}dxdy\Big)^{\frac{1}{2}}.$$
This is a Hilbert space endowed with the following inner product
$$\langle w\,,\,v\rangle=\int\limits_{\mathcal{C}}y^{1-2s}\nabla w\,\nabla v\,dxdy,
\quad \forall \,w\,,v\in H^{s}_{0,L}(\mathcal{C}).$$

\begin{definition}\label{def01}
We say that $u\in H^{s}_{0}(\Omega)$ is a
solution of Equation \eqref{eq01} such that for every function
$\varphi\in H^{s}_{0}(\Omega)$, it holds
$$\int\limits_{\Omega}(-\Delta)^{\frac{s}{2}}u(-\Delta)^{\frac{s}{2}}\varphi dx=\lambda\int\limits_{\Omega}f(x,u)\varphi dx
+\mu\int\limits_{\Omega}|u|^{p-2}u\varphi dx.$$
\end{definition}

 Associated with problem \eqref{eq01} we consider the energy functional
 $$I(u)=\frac{1}{2}\int\limits_{\Omega}|(-\Delta)^{\frac{s}{2}}u|^{2}dx
 -\lambda\int\limits_{\Omega}F(x,u)dx
 -\frac{\mu}{p}\int\limits_{\Omega}|u|^{p}dx.$$

We now conclude the main ingredients of a recently developed technique
which can deal with fractional power of the Laplacian.
To treat the nonlocal problem \eqref{eq01}, we will study a
corresponding extension problem, so that we can investigate
problem \eqref{eq01} by studying a local problem via classical
nonlinear variational methods.

We first define the extension operator and fractional Laplacian for
functions in $H^{s}_{0}(\Omega)$.

\begin{definition}
Given a function $u\in H^{s}_{0}(\Omega)$,
we define its $s$-harmonic extension $w=E_{s}(u)$
to the cylinder $\mathcal{C}$ as a solution of the problem
\begin{equation*}\left\{
\aligned
&div(y^{1-2s}\nabla w)=0,\quad &\text{in}\,\,\mathcal{C},\\
&w=0,&\text{on}\,\,\partial_{L}\mathcal{C},\\
&w=u,&\text{on}\,\,\Omega\times\{0\}.
\endaligned\right.
\end{equation*}
\end{definition}
Following \cite{ll2007}, we can define the fractional Laplacian
operator by the Dirichlet to Neumann map as follows.
\begin{definition}
For any regular function $u(x)$, the fractional Laplacian
$(-\Delta)^{s}$ acting on $u$ is defined by
$$(-\Delta)^{s}u(x)=-\lim\limits_{y\to 0^{+}}y^{1-2s}\frac{\partial w}{\partial y}(x,y),\quad \forall x\in \Omega,\quad y\in (0\,,\,\infty),$$
where $w=E_{s}(u)$.
\end{definition}

From \cite{ll2007} and \cite{bc2013}, the map $E_{s}(\cdot)$ is an
isometry between  $H^{s}_{0}(\Omega)$ and
$H^{s}_{0,L}(\mathcal{C})$.
Furthermore, we have
\begin{itemize}
\item[(i)]$\|(-\Delta)^{s}u\|_{H^{-s}(\Omega)}=
\|u\|_{H^{s}_{0}(\Omega)}
=\|E_{s}(u)\|_{H^{s}_{0,L}(\mathcal{C})}$,
where $H^{-s}(\Omega)$ denotes the dual
space of $H^{s}_{0}(\Omega)$;
\item[(ii)]For any $w\in H^{s}_{0,L}(\mathcal{C})$,
there exists a constant $C$ independent of $w$ such that
$$\|\text{tr}_{\Omega}w\|_{L^{r}(\Omega)}\leq C\|w\|_{H^{s}_{0,L}(\mathcal{C})}$$
holds for every $r\in [2\,,\,\frac{2N}{N-2s}]$.
Moreover, $H^{s}_{0,L}(\mathcal{C})$
is compactly embedded into
$L^{r}(\Omega)$ for $r\in [2\,,\,\frac{2N}{N-2s})$.
\end{itemize}

In the following lemma, we will list some inequalities.
\begin{lemma}\label{le3}
For every $1\leq r\leq \frac{2N}{N-2s}$, and every
$w\in H^{s}_{0,L}(\mathcal{C})$, it holds
$$\Big(\int\limits_{\Omega\times\{0\}}|w|^{r}dx\Big)^{\frac{2}{r}}
\leq C\int\limits_{\mathcal{C}}y^{1-2s}|\nabla w|^{2}dxdy,$$
where constant $C$ depends on $r,\,s,\,N,\,|\Omega|$.
\end{lemma}
\begin{lemma}\label{le4}
For every $w\in H^{s}(\mathbb{R}^{N+1}_{+})$ the sharp fractional Sobolev inequality for $N>2s$
and $s>0$
$$\Big(\int\limits_{\mathbb{R}^{N}}|u(x)|^{2^*_{s}}dx\Big)^{\frac{2}{2^*_{s}}}\leq S
\int\limits_{\mathbb{R}^{N+1}_{+}}y^{1-2s}|\nabla w(x,y)|^{2}dxdy,
$$
which holds with the constant
$$S=\frac{2^{-1}\pi^{-s}\Gamma(s)\Gamma(\frac{N-2s}{2})(\Gamma(N))^{\frac{2s}{N}}}
{\Gamma(\frac{2-2s}{2})\Gamma(\frac{N+2s}{2})(\Gamma(\frac{N}{2}))^{\frac{2s}{N}}},$$
where $u={ tr}_{\Omega}w$.
\end{lemma}
Theorem\ref{th1} and \ref{th2} will be proved by using a recent result
on the existence of at least three critical points by Ricceri \cite{b2009-1,b2009-2}.
For the reader's convenience, we describe it as follows.

 If $X$ is a real Banach space, we can denote by $\mathcal{X}$ the class of all
 function $\phi:\,X\to \mathbb{R}$ possessing the following property: if $\{u_{n}\}\subset X$
is a sequence converging weakly to $u\in X$ and
$\liminf\limits_{n\to \infty}\phi(u_{n})\leq \phi(u)$, then $\{u_{n}\}$ has a subsequence
converging strongly to $u$.

\begin{theorem}\label{th3}
Let $X$ be a separable and reflexive real Banach space; $I\subseteq\mathbb{R}$ an interval;
$\Phi:\,X\to \mathbb{R}$ a sequentially weakly lower semi-continuous $C^1$ functional,
belonging to $\mathcal{X}$, bounded on each bounded subset of $X$ and whose derivative
admits a continuous inverse on $X^*$; $J:\,X\to \mathbb{R}$ a $C^1$ functional with
compact derivative. Assume that, for each $\lambda\in I$, the functional $\Phi-\lambda J$
is coercive and has a strict local, not global minimum, say $\hat{u}_{\lambda}$.
Then, for each compact interval $[a,b]\subseteq I$ for which
$\sup\limits_{\lambda\in[a,b]}(\Phi(\hat{u}_{\lambda})-\lambda J(\hat{u}_{\lambda}))<+\infty$,
there exists $\gamma>0$ with the following property: for every $\lambda\in[a,b]$ and every
 $C^1$ functional $\Psi:\,X\to \mathbb{R}$ with compact derivative,
 there exists $\delta_{0}>0$ such that, for each $\mu \in[0\,,\,\delta_{0}]$, the equation
$$\Phi'(u)=\lambda J'(u)+\mu \Psi'(u)$$
has at least three solutions whose norm are less than $\gamma$.
\end{theorem}

\section{Proof of the main results}
Let
$$\Psi(u)=\frac{1}{2}\int\limits_{\mathcal{C}}y^{1-2s}|\nabla w|^{2}dxdy,
\quad J(w)=\int\limits_{\Omega\times\{0\}}F(x,w)dx.$$
Obviously, the condition $(f_3)$ implies
\begin{equation*}\label{eq10}
\theta=\sup\limits_{\Psi(w)\not\neq 0}\frac{J(w)}{\Psi(w)}>0.
\end{equation*}

\begin{lemma}\label{le1}
Let $f$ satisfy $(f_1)-(f_3)$. Then for every $\lambda\in (0\,,\,\infty)$, the functional $\Phi-\lambda J$
is sequentially weakly lower continuous and coercive on $H^{s}_{0,L}(\mathcal{C})$, and has a global minimizer $w_{\lambda}$.

\end{lemma}

\begin{proof}
By $(f_1)$ and $(f_3)$, for any  $\varepsilon>0$, there exist $M_{0}>0$ and $C_{1}>0$
such that, for all $s\in H^{s}_{0,L}(\mathcal{C})$,
$$|f(x,s)|\leq \varepsilon|s|,\quad \forall\,|s|\geq M_{0}, $$
and
$$|f(x,s)|\leq C_{1},\quad \forall\,|s|\leq M_{0}+1.$$
So, for any $w\in H^{s}_{0,L}(\mathcal{C})$, we have
\begin{equation}\label{eq11}
|f(x,w)|\leq C_{1}+\varepsilon|w|,
\end{equation}
which implied that
\begin{equation*}\label{eq12}
|F(x,w)|\leq C_{1}|w|+\frac{\varepsilon}{2}|w|^{2}.
\end{equation*}
Thus, for all $w\in H^{s}_{0,L}(\mathcal{C})$, we obtain
\begin{equation*}\label{eq13}
\aligned
\Phi(w)-\lambda J(w)
&=\frac{1}{2}\int\limits_{\mathcal{C}}y^{1-2s}|\nabla w|^{2}dxdy
-\lambda\int\limits_{\Omega\times\{0\}}F(x,w)dx\\
&\geq \frac{1}{2}\int\limits_{\mathcal{C}}y^{1-2s}|\nabla w|^{2}dxdy
-\lambda\int\limits_{\Omega\times\{0\}}(C_{1}|w|+\frac{\varepsilon}{2}|w|^{2})dx\\
&=\frac{1}{2}\| w\|^{2}_{H^{s}_{0,L}(\mathcal{C})}
-\lambda\frac{\varepsilon}{2}\int\limits_{\Omega\times\{0\}}|w|^{2}dx
-\lambda C_{1}\int\limits_{\Omega\times\{0\}}|w|dx\\
&\geq \Big(\frac{1}{2}-\lambda\frac{\varepsilon C_{2}}{2}\Big)\| w\|^{2}_{H^{s}_{0,L}(\mathcal{C})}
-\lambda C_{1}C_{3}\|w\|_{H^{s}_{0,L}(\mathcal{C})},\\
\endaligned\end{equation*}
where constants $C_{2}>0$, $C_{3}>0$.
Let $\varepsilon>0$  small enough such that $\frac{1}{2}-\lambda\frac{\varepsilon C_{2}}{2}>0$, then we have
$$\Phi(w)-\lambda J(w)\to +\infty\quad \text{as}\quad \|w\|_{H^{s}_{0,L}(\mathcal{C})}
\to \infty.$$
Hence, $\Phi-\lambda J$ is coercive.

Moreover, from the embedding $H^{s}_{0,L}(\mathcal{C})\hookrightarrow L^{r}(\Omega)$ ($1\leq r<2^*_{s}$)
is compact and \eqref{eq11}, $J$ is weakly continuous.
Obviously, $$\Phi(u)=\frac{1}{2}\int\limits_{\mathcal{C}}y^{1-2s}|\nabla w|^{2}dxdy=
\frac{1}{2}\|w\|^{2}_{H^{s}_{0,L}(\mathcal{C})}$$
is weakly lower semi-continuous on $H^s_{0,L}(\mathcal{C})$. We can deduce that $\Phi-\lambda J$ is a sequentially weakly lower semi-continuous.
So, $\Phi-\lambda J$ has a global minimizer $w_{\lambda}\in H^s_{0,L}(\mathcal{C})$. The proof is completed.
\end{proof}

Next, we will show that $\Phi-\lambda J$ has a strictly local,
not global minimizer for some $\lambda$,
when $f$ satisfies $(f_1)-(f_3)$.

\begin{lemma}\label{le2}
Let $f$ satisfy $(f_1)-(f_3)$. Then
\begin{itemize}
\item[{\em (i)}] 0 is a strict local minimizer of the functional $\Phi-\lambda J$ for $\lambda\in (0\,,\,+\infty)$.
\item[{\em (ii)}] $w_{\lambda}\not\neq 0$, i.e., 0 is not the global minimizer $w_{\lambda}$ for
$\lambda\in (\frac{1}{\theta}\,,\,+\infty)$, where $w_{\lambda}$ is given by Lemma \ref{le1}.
\end{itemize}
\end{lemma}

\begin{proof}
Firstly, we prove that
\begin{equation*}\label{eq14}
\lim\limits_{\|w\|_{H^{s}_{0,L}(\mathcal{C})}\to 0}\frac{J(w)}{\Phi(w)}=0,\quad \forall w\in H^{s}_{0,L}(\mathcal{C}).
\end{equation*}
In fact, by $(f_2)$, for any $ \varepsilon>0$, there exists a $\delta>0$, such that
\begin{equation}\label{eq15}
|f(x,w)|\leq \varepsilon|w|,\quad |w|<\delta.
\end{equation}
Considering the inequality \eqref{eq15}, $(f_1)$ and $(f_3)$, there exists
$r\in (1\,,\,2^*_{s}-1)$ such that
\begin{equation}\label{eq16}
|f(x,w)|\leq \varepsilon|w|+|w|^{r}.
\end{equation}
Then from Lemma \ref{le3}, there exist $C_{4}, C_{5}>0$, such that
\begin{equation*}\label{eq17}
|J(w)|\leq \varepsilon\, C_{4}\|w\|^{2}_{H^{s}_{0,L}(\mathcal{C})}+C_{5}\,\|w\|^{r+1}_{H^{s}_{0,L}(\mathcal{C})}.
\end{equation*}
This implies
\begin{equation*}\label{eq17}
\lim\limits_{\|w\|_{H^{s}_{0,L}(\mathcal{C})}\to 0}\frac{J(w)}{\Phi(w)}=0.
\end{equation*}

Next, we will prove (i) and (ii).
\begin{itemize}
\item[(i)]For $\lambda\in (0\,,\,+\infty)$, since
$\lim\limits_{\|w\|_{H^{s}_{0,L}(\mathcal{C})}\to 0}\frac{J(w)}{\Phi(w)}=0< \frac{1}{\lambda}$
and
$\Phi(w)>0$ for each $w\not\neq 0$ in some neighborhood $U$ of 0, there exists a neighborhood
$V\subseteq U$ of 0 such that $$\Phi(w)-\lambda J(w)>0,\quad\forall\,\, w\in V\setminus\{0\}.$$
Hence, 0 is a strict local minimum of $\Phi-\lambda J$.
\item[(ii)] For $\lambda\in (\frac{1}{\theta}\,,\,+\infty)$, from the definition of $\theta$,
there exists $\hat{w}\in H^{s}_{0,L}(\mathcal{C})$such that $\Phi(\hat{w})>0$, $J(\hat{w})>0$
and $\frac{J(\hat{w})}{\Phi(\hat{w})}>\frac{1}{\lambda}$. So we have
 $$  \Phi(\hat{w})-\lambda J(\hat{w})<0=\Phi(0)-\lambda J(0).$$
This yields $0$ is not a global minimum of $\Phi-\lambda J$.
\end{itemize}
This completes the proof.
\end{proof}

Let $K>0$ be a real number, whose value will be fixed latter.
Define the truncation of $|w|^{p-2}w$ with $p>2^{*}_{s}$, be given by
\begin{equation*}\label{eq18}
g_{K}(w)=\left\{\aligned
&|w|^{p-2}w,\quad   &\text{if}\quad 0\leq|w|\leq  K,\\
&K^{p-q}|w|^{q-2}w, \quad   &\text{if}\quad |w|> K,
\endaligned\right.
\end{equation*}
where $q\in(2\,,\,2_{s}^{*})$. Then $g_{K}(w)$ satisfies
$$|g_{K}(w)|\leq K^{p-q}|w|^{q-1},$$
for $K$ large enough.
Then, we study the truncated problem
\begin{equation}\label{eq19}
\left\{
\aligned
&div(y^{1-2s}\nabla w)=0,\quad &\text{in}\,\,\mathcal{C},\\
&w=0,&\text{on}\,\,\partial_{L}\mathcal{C},\\
&\partial^{s}_{\nu}w=\lambda f(x,w)+\mu g_{K}(w),\quad &\text{in}\,\, \Omega\times\{0\},\\
\endaligned\right.
\end{equation}

We say that $w\in H^{s}_{0,L}(\mathcal{C})$ is a weak solution of the  problem \eqref{eq19} if
\begin{equation}\label{eq19*}
\int\limits_{\mathcal{C}}y^{1-2s}\nabla w\cdot\nabla \varphi dxdy
=\lambda\int\limits_{\Omega\times\{0\}} f(x,w)\varphi dx +\mu\int\limits_{\Omega\times\{0\}} g_{K}(w)\varphi dx
\end{equation}
for every $\varphi\in H^{s}_{0,L}(\mathcal{C})$.

Let
$$\Psi(u)=\int\limits_{\Omega\times\{0\}}G_{K}(w)dx,$$
where $G_{K}(w)=\int^{u}_{0}g_{K}(t)dt$. So
 from $|g_{K}(w)|\leq K^{p-q}|w|^{q-1}$, $2<q<2^*_{s}$, we get that
 $g_{K}(w)$ is a super-linear function with subcritical growth, then
$\Psi(u)$ has a compact derivative in $H^{s}_{0,L}(\mathcal{C})$. Moreover,
for each compact interval $[a,b]\subset(\frac{1}{\theta},+\infty)$, $\lambda\in[a\,,\,b]$.
From \eqref{eq16}, we have
then $J(w)$ has a compact derivative in $H^{s}_{0,L}(\mathcal{C})$ too.
Therefore, it is easy to see that the functional
$$\mathcal{E}(w)=\Phi(w)-\lambda J(w)-\mu\Psi(w)\quad \forall \,w\in H^{s}_{0,L}(\mathcal{C})$$ is
 $C^{1}$ and its derivative is given by

$$\langle \mathcal{E}'(w),\varphi\rangle
=\int\limits_{\mathcal{C}}y^{1-2s}\nabla w\nabla \varphi dxdy
-\lambda\int\limits_{\Omega\times\{0\}}f(x,w)\varphi dx
-\mu\int\limits_{\Omega\times\{0\}}g_{K}(w)\varphi dx, $$
for all $\varphi\in H^{s}_{0,L}(\mathcal{C})$.

By Lemma \ref{le1} and Lemma \ref{le2},
all the hypotheses of Theorem \ref{th3} are satisfied. So there exists $\gamma >0$
with the following property: for every $\lambda\in[a,b]\subset(\frac{1}{\theta},+\infty)$,
there exists $\delta_{0}>0$, such that for $\mu\in[0,\delta_{0}]$,
 the problem \eqref{eq19} has at least three solutions $w_{0}$, $w_{1}$ and $w_{2}$ in
  $H^{s}_{0,L}(\mathcal{C})$ and
 $$ \| w_{k}\|_{H^{s}_{0,L}(\mathcal{C})}\leq\gamma,\quad k=0,1,2,$$
  where $\gamma$ depends on $\lambda$, but does depend on $\mu$
 or $K$.

 If the three solutions $w_{k}$, $k=0,1,2$, satisfy
 \begin{equation}\label{eq20}
|w_{k}|\leq K, \quad \text{a.e.}\,\, (x,y)\in\Omega\times(0\,,\,\infty),\quad k=0,1,2.
\end{equation}
Then in the view of the definition $g_{K}$,
we have $g_{K}(x,w)=\mu|w|^{p-2}w$ and
therefore $w_{k}$, $k=0,1,2,$ are also solutions of the original problem \eqref{eq02}.
Thus, in order to prove Theorem \ref{th1},
it suffices to show that exists $\delta_{0}>0$, such that for $\mu\in[0,\delta_{0}]$,
the solutions obtained by Theorem \ref{th3} satisfy the inequality \eqref{eq20}.

{\bf Proof of theorem \ref{th1}.}
Our aim is to show that exits $\delta_{0}>0$,
such that for $\mu\in[0,\delta_{0}]$, the solution $w_{k}$, $k=0,1,2$,
 satisfy the inequality \eqref{eq20}. To save notation, we will denote $w:=w_{k}$, $k=0,1,2$.

Set $w_{+}=\max\{w\,,\,0\},w_{-}=-\min\{w\,,\,0\}$. Then $|w|=w_{+}+w_{-}$.
We can argue with the positive and negation part of $w$ separately.

We first deal with $w_{+}$. For each $L>0$, we define the following functions
\begin{equation*}\label{eq21}
w_{L}=\left\{\aligned
&w_{+},\quad  &\text{if} \,\,\,\, w_{+}\leq L,\\
&L, \quad  &\text{if}\,\, \,\,w_{+}> L.
\endaligned\right.
\end{equation*}

For $\beta>1$ to be determined, we choose in \eqref{eq19*} that
$$\varphi=w_{L}^{2(\beta-1)}w_{+},$$
and since
$$\nabla\varphi=w_{L}^{2(\beta-1)}\nabla w_{+}+2(\beta-1)w_{L}^{2(\beta-1)-1}w_{+}\nabla w_{L},$$
we obtain
\begin{equation}\label{eq23}
\aligned
&\int\limits_{\mathcal{C}}y^{1-2s}\nabla w\nabla\varphi dxdy\\
&=\int\limits_{\mathcal{C}}y^{1-2s}(\nabla (w_{+}- w_{-}))\nabla (w_{L}^{2(\beta-1)}w_{+})dxdy\\
&=\int\limits_{\mathcal{C}}y^{1-2s}(\nabla w_{+}-\nabla w_{-})(w_{L}^{2(\beta-1)}\nabla w_{+}+
2(\beta-1)w_{L}^{2(\beta-1)-1}w_{+}\nabla w_{L})dxdy\\
&=\int\limits_{\mathcal{C}}y^{1-2s}\Big(|\nabla w_{+}|^{2}w_{L}^{2(\beta-1)}
+2(\beta-1)w_{L}^{2(\beta-1)-1}w_{+}\nabla w_{L}\nabla w_{+}\Big)dxdy\\
&=\int\limits_{\mathcal{C}}y^{1-2s}w_{L}^{2(\beta-1)}|\nabla w_{+}|^{2}dxdy
+2(\beta-1)\int\limits_{\mathcal{C}}y^{1-2s}w_{L}^{2(\beta-1)-1}w_{+}\nabla w_{L}\nabla w_{+}dxdy.
\endaligned
\end{equation}
From the definition of $w_{L}$, we have
\begin{equation}\label{eq24}
\aligned
&2(\beta-1)\int\limits_{\mathcal{C}}y^{1-2s}w_{L}^{2(\beta-1)-1}w_{+}\nabla w_{L}\nabla w_{+}dxdy\\
&=2(\beta-1)\int\limits_{\{w_{+}<L\}}y^{1-2s}\,w_{L}^{2(\beta-1)-1}w_{+}\nabla w_{L}\nabla w_{+}dxdy\\
&=2(\beta-1)\int\limits_{\{w_{+}<L\}}y^{1-2s}\,w_{+}^{2(\beta-1)}|\nabla w_{+}|^{2}dxdy\\
&\geq 0.\\
\endaligned
\end{equation}

Set
$$h_{K}(x,w)=\lambda f(x,w)+\mu g_{K}(x,w),\quad \forall w\in H^{s}_{0,L}(\mathcal{C}).$$
From \eqref{eq16} and $|g_{K}(x,w)|\leq K^{p-q}|w|^{q-1}$, we can choose constant $C_6>0$
such that
\begin{equation}\label{eq26}
|h_{K}(x,w)|\leq C_{6}|w|+\mu K^{p-q}|w|^{q-1}.
\end{equation}
We deduce from \eqref{eq19*}, \eqref{eq23}, \eqref{eq24} and \eqref{eq26} for $\beta>1$ that
\begin{equation}\label{eq27}
\aligned
\int\limits_{\mathcal{C}}y^{1-2s} w_{L}^{2(\beta-1)}|\nabla w_{+}|^{2}dxdy
&=\int\limits_{\Omega\times\{0\}} h_{K}(x,w)\varphi dx
\leq\int\limits_{\Omega\times\{0\}}|h_{K}(x,w)\varphi|dx\\
&\leq\int\limits_{\Omega\times\{0\}}\Big(C_{6}|w|+\mu K^{p-q}|w|^{q-1}\Big)w_{L}^{2(\beta-1)}w_{+}dx\\
&=\int\limits_{\Omega\times\{0\}}\Big(C_{6}(w_{+}+w_{-})
+\mu K^{p-q}\,(w_{+}+w_{-})^{q-1}\Big)w_{L}^{2(\beta-1)}w_{+}dx\\
&=\int\limits_{\Omega\times\{0\}}\Big(C_{6}\,w_{+}^{2}w_{L}^{2(\beta-1)}
+\mu K^{p-q}\, w_{+}^{q}\,w_{L}^{2(\beta-1)}\Big)dx.
\endaligned
\end{equation}

Let $\hat{w}_{L}=w_{+}\,w_{L}^{\beta-1}$, we have
$$\nabla \hat{w}_{L}=w_{L}^{\beta-1}\,\nabla w_{+}+(\beta-1)w_{+}\,w_{L}^{\beta-2}\,\nabla w_{L}.$$
By the Sobolev embedding  theorem,
\begin{equation}\label{eq28}
\aligned
&\Big(\int\limits_{\Omega\times\{0\}}|\hat{w}_{L}|^{2^{*}_{s}}dx\Big)^{\frac{2}{2^{*}_{s}}}
\leq S\,\int\limits_{\mathcal{C}}y^{1-2s}|\nabla \hat{w}_{L}|^{2}dxdy\\
&=C\int\limits_{\mathcal{C}}y^{1-2s}|w_{L}^{\beta-1}\nabla w_{+}+
(\beta-1)w_{+}\,w_{L}^{\beta-2}\nabla w_{L}|^{2}dxdy\\
&\leq 2\,S\,\Big(\int\limits_{\mathcal{C}}y^{1-2s}|(\beta-1)w_{+}w_{L}^{\beta-2}\nabla w_{L}|^{2}dxdy
+\int\limits_{\mathcal{C}}y^{1-2s}|w_{L}^{\beta-1}\nabla w_{+}|^{2}dxdy\Big)\\
&=2\,S\,\Big(\int\limits_{\mathcal{C}}y^{1-2s}(\beta-1)^{2}|w_{L}|^{2(\beta-1)}|\nabla w_{+}|^{2}dxdy
+\int\limits_{\mathcal{C}}y^{1-2s}|w_{L}^{\beta-1}\nabla w_{+}|^{2}dxdy)\\
&\leq 2S\Big((\beta-1)^{2}\int\limits_{\mathcal{C}}y^{1-2s}w_{L}^{2(\beta-1)}|\nabla w_{+}|^{2}dxdy
+\int\limits_{\mathcal{C}}y^{1-2s}w_{L}^{2(\beta-1)}|\nabla w_{+}|^{2}dxdy\Big)\\
&=2S\,\Big((\beta-1)^{2}+1\Big)\int\limits_{\mathcal{C}}y^{1-2s}w_{L}^{2(\beta-1)}|\nabla w_{+}|^{2}dxdy\\
&=2S\,\beta^{2}\Big((\frac{\beta-1}{\beta})^{2}+\frac{1}{\beta^{2}}\Big)
\int\limits_{\mathcal{C}}y^{1-2s}w_{L}^{2(\beta-1)}|\nabla w_{+}|^{2}dxdy,
\endaligned
\end{equation}
where $S>0$ is the Sobolev embedding constant.

Since $\beta>1$, we have $\frac{1}{\beta^{2}}<1$
and $(\frac{\beta-1}{\beta})^{2}<1$. From \eqref{eq27} and \eqref{eq28}, we get
\begin{equation}\label{eq30}\aligned
&2S\,\beta^{2}\Big((\frac{\beta-1}{\beta})^{2}+\frac{1}{\beta^{2}}\Big)
\int\limits_{\mathcal{C}}y^{1-2s}w_{L}^{2(\beta-1)}|\nabla w_{+}|^{2}dxdy\\
&< 4\,S\beta^{2}\int\limits_{\mathcal{C}}y^{1-2s}w_{L}^{2(\beta-1)}|\nabla w_{+}|^{2}dxdy\\
&\leq 4\,S\beta^{2}\int\limits_{\Omega\times\{0\}}(C_{6}\,w^{2}_{+}w^{2(\beta-1)}_{L}
+\mu K^{p-q}w_{+}^{q}w^{2(\beta-1)}_{L})dx.\\
\endaligned
\end{equation}

From the Sobolev embedding $H^{s}_{0,L}(\mathcal{C})\hookrightarrow L^{2^{*}_{s}}(\Omega)$ and
$\|w_{+}\|_{H^{s}_{0,L}(\mathcal{C})}\leq \gamma$, we have
\begin{equation}\label{eq31}
\Big(\int\limits_{\Omega\times\{0\}}|w_{+}|^{2^*_{s}}dx\Big)^{\frac{2}{2^*_{s}}}
\leq S\, \int\limits_{\mathcal{C}}y^{1-2s}|\nabla w_{+}|^{2}dxdy\leq S\,\gamma.
\end{equation}

Let $t=\frac{2\,2^{*}_{s}}{2^{*}_{s}-q+2}$.
Since $w^{q}_{+}\,w^{2(\beta-1)}_{L}=w^{q}_{+}\hat{w}^{2}_{L}w_{+}^{-2}
=w^{q-2}_{+}\hat{w}^{2}_{L}$
and $\hat{w}_{L}^{2}=w^{2}_{+}w^{2(\beta-1)}_{L}$,
we can use the H\"{o}lder's inequality, \eqref{eq28}, \eqref{eq30} and \eqref{eq31} to conclude that,
whenever $\hat{w}_{L}(\cdot,0)\in L^{t}(\Omega)$, it holds
\begin{equation*}\label{eq32}
\aligned
&\Big(\int\limits_{\Omega\times\{0\}}|\hat{w}_{L}|^{2^*_{s}}dx\Big)^{\frac{2}{2^*_{s}}}\\
&\leq 4S\,\beta^{2}\int\limits_{\Omega\times\{0\}}\Big(C\,w^{2}_{+}w^{2(\beta-1)}_{L}
+\mu K^{p-q}w^{q}_{+}w^{2(\beta-1)}_{L}\Big)dx,\\
&=4 S\,\beta^{2}\Big(C\int\limits_{\Omega\times\{0\}}\hat{w}^{2}_{L}dx
+\mu K^{p-q}\int\limits_{\Omega\times\{0\}}w^{q-2}_{+}\hat{w}^{2}_{L}dx\Big),\\
&\leq4S\,\beta^{2} \Big[|\Omega|^{\frac{q-2}{2^{*}_{s}}}
\Big(\int\limits_{\Omega\times\{0\}}\hat{w}^{t}_{L}dx\Big)^{\frac{2}{t}}
+\mu K^{p-q}\Big(\int\limits_{\Omega\times\{0\}}|w_{+}|^{2^{*}_{s}}dx\Big)^{\frac{q-2}{2^{*}_{s}}}
(\int\limits_{\Omega\times\{0\}}\hat{w}^{t}_{L}dx\Big)^{\frac{2}{t}}\Big]\\
&\leq 4S\,\beta^{2}\Big(|\Omega|^{\frac{q-2}{2^{*}_{s}}}+\mu K^{p-q}(S\,\gamma)^{\frac{q-2}{2}}\Big)\|\hat{w}_{L}\|^{2}_{L^{t}}.
\endaligned
\end{equation*}

Set $\beta:=\frac{2^{*}_{s}}{t}=1+\frac{2^*_{s}-q}{2}>1$.
By the definition of $w_{L}$, we have $w_{L}\leq w_{+}$, then
we conclude that $\hat{w}_{L}(\cdot,0)\in L^{t}(\Omega)$,
whenever $(w_{+}(\cdot,0))^{\beta}\in L^{t}(\Omega)$.
If this is the case, it follow from the above inequality that
\begin{equation*}\label{eq33}\aligned
\Big(\int\limits_{\Omega\times\{0\}}|\hat{w}_{L}|^{2^*_{s}}dx\Big)^{\frac{2}{2^*_{s}}}
&=\Big(\int\limits_{\Omega\times\{0\}}w_{L}^{2^{*}_{s}(\beta-1)}w_{+}^{2^{*}_{s}}dx\Big)^{\frac{2}{2^{*}_{s}}}\\
&\leq 4S\beta^{2}\Big(|\Omega|^{\frac{q-2}{2^{*}_{s}}}
+\mu K^{p-q}(S\,\gamma)^{\frac{q-2}{2}}\Big)
\Big(\int\limits_{\Omega\times\{0\}}|w_{L}^{\beta-1}\,w_{+}|^{t}dx\Big)^{\frac{2}{t}}.
\endaligned
\end{equation*}
By Fatou's Lemma in the variable $L$, we get
\begin{equation*}\label{eq33*}
\Big(\int\limits_{\Omega\times\{0\}}w_{+}^{2^{*}_{s}\,\beta}dx\Big)^{\frac{2\beta}{2^{*}_{s}\beta}}
\leq 4\,S\beta^{2}C_{\mu,K}\Big(\int\limits_{\Omega\times\{0\}}|w_{+}|^{t\beta }dx\Big)^{\frac{2\beta}{t\beta}},
\end{equation*}
i.e.,
\begin{equation}\label{eq33**}
\Big(\int\limits_{\Omega\times\{0\}}w_{+}^{2^{*}_{s}\,\beta}dx\Big)^{\frac{1}{2^{*}_{s}\beta}}
\leq \Big(4S\,C_{\mu,K}\Big)^{\frac{1}{2\beta}}\beta^{\frac{1}{\beta}}
\Big(\int\limits_{\Omega\times\{0\}}|w_{+}|^{t\beta }dx\Big)^{\frac{1}{t\beta}},
\end{equation}
where $C_{\mu,K}=|\Omega|^{\frac{q-2}{2^{*}_{s}}}+\mu K^{p-q}(S\,\gamma)^{\frac{q-2}{2}}$.

Since $\beta=\frac{2^{*}_{s}}{t}>1$ and $w_{+}(\cdot,0)\in L^{2^{*}_{s}}(\Omega)$,
the inequality \eqref{eq33**} holds for this choice of $\beta$.
Therefore, from $\beta^{2}\,t=\beta 2^{*}_{s}$, we have that the inequality \eqref{eq33**}
 also holds with $\beta$ replaced by $\beta^{2}$. Hence
 \begin{equation*}\label{eq35}
 \aligned
 \Big(\int\limits_{\Omega\times\{0\}}w_{+}^{2^{*}_{s}\,\beta^{2}}dx\Big)^{\frac{1}{2^{*}_{s}\beta^{2}}}
 &\leq \Big(4S\,C_{\mu,K}\Big)^{\frac{1}{2\beta^{2}}}(\beta^{2})^{\frac{1}{\beta^{2}}}
\Big(\int\limits_{\Omega\times\{0\}}|w_{+}|^{t\beta^{2} }dx\Big)^{\frac{1}{t\beta^{2}}}\\
&=\Big(4S\,C_{\mu,K}\Big)^{\frac{1}{2\beta^{2}}}(\beta^{2})^{\frac{1}{\beta^{2}}}
\Big(\int\limits_{\Omega\times\{0\}}|w_{+}|^{2^*_{s}\beta }dx\Big)^{\frac{1}{2^{*}_{s}\beta}}\\
&\leq \Big(4S\,C_{\mu,K}\Big)^{\frac{1}{2\beta^{2}}}(\beta^{2})^{\frac{1}{\beta^{2}}}
\Big(4S\,C_{\mu,K}\Big)^{\frac{1}{2\beta}}\beta^{\frac{1}{\beta}}
\Big(\int\limits_{\Omega\times\{0\}}|w_{+}|^{t\beta }dx\Big)^{\frac{1}{t\beta}}\\
&= \Big(4S\,C_{\mu,K}\Big)^{\frac{1}{2}(\frac{1}{\beta^{2}}+\frac{1}{\beta})}
\beta^{\frac{2}{\beta^{2}}+\frac{1}{\beta}}
\Big(\int\limits_{\Omega\times\{0\}}|w_{+}|^{t\beta }dx\Big)^{\frac{1}{t\beta}}
\endaligned
\end{equation*}
By iterating this process and $\beta\,t=2^{*}_{s}$, we obtain
 \begin{equation}\label{eq36}
 \Big(\int\limits_{\Omega\times\{0\}}w_{+}^{2^{*}_{s}\,\beta^{m}}dx\Big)^{\frac{1}{2^{*}_{s}\beta^{m}}}
\leq
\Big(4S\,C_{\mu,K}\Big)^{\frac{1}{2}(\frac{1}{\beta^{m}}+\cdot\cdot\cdot+\frac{1}{\beta^{2}}+\frac{1}{\beta})}
\beta^{\frac{m}{\beta^{m}}+\cdot\cdot\cdot+\frac{2}{\beta^{2}}+\frac{1}{\beta}}
\Big(\int\limits_{\Omega\times\{0\}}|w_{+}|^{2^*_{s} }dx\Big)^{\frac{1}{2^*_{s}}}.
\end{equation}

Taking the limit as $m\rightarrow\infty$ in \eqref{eq36}, we have
\begin{equation*}\label{eq37}
\|w_{+}\|_{L^{\infty}}
\leq(4S\,C_{\mu,K})^{\theta_{1}}\beta^{\theta_{2}}\|w_{+}\|_{L^{2^*_{s}}}
\leq(4S\,C_{\mu,K})^{\theta_1}\beta^{\theta_2} (S\,\gamma)^{\frac{1}{2}},
\end{equation*}
where $\theta_1=\frac{1}{2}\sum\limits^{\infty}_{m=1}\frac{1}{\beta^{m}}$,
$\theta_2=\sum\limits^{\infty}_{m=1}\frac{m}{\beta^{m}}$ and $\beta>1$.

Next, we will find some suitable value of $K$ and $\mu$,
such that the inequality
\begin{equation}\label{eq38}
(4\,S\,C_{\mu,K})^{\theta_1}\beta^{\theta_2} (S\,\gamma)^{\frac{1}{2}}\leq\frac{K}{2}
\end{equation}
holds. From \eqref{eq38}, we get
\begin{equation*}\label{eq39}
C_{\mu,K}=|\Omega|^{\frac{q-2}{2^{*}_{s}}}+\mu K^{p-q}(S\,\gamma)^{\frac{q-2}{2}}
\leq\frac{1}{4S}\Big(\frac{K}{2(\gamma S)^{\frac{1}{2}}\beta^{\theta_2}}\Big)^{\frac{1}{\theta_1}}.
\end{equation*}
Then, choose $K$ to satisfy the inequality
\begin{equation*}\label{eq40}
\frac{1}{4\,S}\Big(\frac{K\,}{2(S\gamma)^{\frac{1}{2}} \beta^{\theta_2}}\Big)^{\frac{1}{\theta_1}}
-|\Omega|^{\frac{q-2}{2^{*}_{s}}}>0,
\end{equation*}
and fix $\mu_{0}$ such that
\begin{equation*}\label{eq41}
0<\mu_{0}<\mu':=\frac{1}{K^{p-q}(S\gamma)^{\frac{q-2}{2}}}\Big{\{}\frac{1}{4S}
\Big(\frac{K}{2(S\gamma)^{\frac{1}{2}}\beta^{\theta_2}}\Big)
^{\frac{1}{\theta_1}}
-|\Omega|^{\frac{q-2}{2^{*}_{s}}}\Big{\}}.
\end{equation*}
Thus, we obtain \eqref{eq38} for $\mu\in[0,\mu_{0}]$, i.e.,
\begin{equation}\label{eq42}
\|w_{+}\|_{L^{\infty}}\leq\frac{K}{2},\quad \text{for}\,\, \mu\in[0,\mu_{0}].
\end{equation}
Similarly, we can also have the estimate for the $w_{-}$, i.e.,
\begin{equation}\label{eq43}
\|w_{-}\|_{L^{\infty}}\leq\frac{K}{2},\quad \text{for }\,\, \mu\in[0,\mu_{0}].
\end{equation}

Now, let $\delta=\min\{\delta_{0},\mu_{0}\}$. For each $\mu\in [0,\delta]$, from \eqref{eq42}, \eqref{eq43}
and $|w|=w_{+}+w_{-}$, we have
\begin{equation*}\label{eq44}
\|w\|_{L^{\infty}}\leq K,\quad \text{for }\,\, \mu\in[0,\mu_{0}].
\end{equation*}

Considering this fact and $w:=w_{k}$, $k=1,2,3$ we get
$$\|w_{k}\|_{L^{\infty}}\leq K,\,\,k=0,1,2,\quad  \text{for}\,\, \mu\in[0,\delta].$$
Therefore, we obtain the inequality \eqref{eq20}.
The proof is completed.

{\bf Proof of theorem \ref{th2}.}
In fact,the truncation of $g_{K}(x,s)$ can be given by
\begin{equation*}\label{eq46}
g_{K}(x,s)=\left\{\aligned
&g(x,s)\quad  &\text{if}\, \,|s|\leq K\\
&\min\{g(x,s),C_{0}(1+K^{p-q}|s|^{q-2}s)\}\quad  &\text{if}\,\,|s|>K
\endaligned\right.
\end{equation*}
where $q\in(2\,,\,2^{*}_{s})$, Then $g_{K}$ satisfies
\begin{equation*}\label{eq46}
|g_{K}(x,s)|\leq C_{0}(1+K^{p-q}|s|^{q-2}),\quad \forall s\in \mathbb{R}.
\end{equation*}

Let $h_{K}(x,w)=\lambda f(x,w)+\mu g_{K}(x,w)$, $\forall w\in H^{s}_{0,L}(\mathcal{C})$.
The truncated problems associated to $h_{K}$
\begin{equation}\label{eq47}
\left\{
\aligned
&div(y^{1-2s}\nabla w)=0,\quad &\text{in}\,\,\mathcal{C},\\
&w=0,&\text{on}\,\,\partial_{L}\mathcal{C},\\
&\partial^{s}_{\nu}w=h_{K}(x,w),\quad &\text{in}\,\, \Omega\times\{0\}.\\
\endaligned\right.
\end{equation}

Similar the proof of the Theorem \ref{th1}, using Theorem \ref{th3}
we can prove that there exists $\delta>0$ such that the solutions $w$ for the truncated
problems \eqref{eq47} satisfy $\|w\|_{L^{\infty}}\leq K$ for $\mu\in[0,\delta]$;
and in view of the definition $g_K$, we have
$$h_{K}(x,w)=\lambda f(x,w)+\mu g(x,w).$$
Therefore $w:=w_{k}$, $k=0,1,2$,
are also solutions of the original problem \eqref{eq04}.

\end{document}